\documentclass[12pt]{amsart}
\usepackage{amsmath}
\usepackage{amsxtra}
\usepackage{amstext}
\usepackage{amssymb}
\usepackage{amsthm}
\usepackage{latexsym}
\usepackage{dsfont} 
\usepackage{verbatim}
\usepackage{tabls}
\usepackage{graphicx}
\usepackage{rotating}
\usepackage{caption}

\theoremstyle{plain}
\newtheorem{thm}{Theorem}[section]

\newtheorem{lem}[thm]{Lemma}
\newtheorem{prop}[thm]{Proposition}

\theoremstyle{definition}
\newtheorem{defn}[thm]{Definition}

\newtheorem*{goal}{MAIN GOAL}
\newtheorem{cla}[thm]{Claim}

\numberwithin{equation}{section}
\def\R1{\widetilde{R}}
\def\T1{\widetilde{T}}
\def\dist{\operatorname{dist}}
\def\supp{\operatorname{supp}}
\def\Lip{\operatorname{Lip}}
\def\eps{\varepsilon}

\def\R{\mathbb{R}}

\def\dy{\mathcal{D}}
\def\dysel{\mathcal{D}_{\operatorname{sel}}}
\def\dyselA{\widehat{\mathcal{D}}_{\operatorname{sel}}}
\def\wh{\widehat}
\def\wt{\widetilde}

\def\XXint#1#2#3{{\setbox0=\hbox{$#1{#2#3}{\int}$}
     \vcenter{\hbox{$#2#3$}}\kern-.5\wd0}}

\begin{document}

\title[The boundedness of $s$-dimensional CZOs]
{On the boundedness of non-integer dimension Calder\'{o}n-Zygmund Operators with antisymmetric kernels}

\author[B. Jaye]{Benjamin Jaye}
\author[F. Nazarov]{Fedor Nazarov}
\address{Department of Mathematical Sciences, Kent State University, Kent, Ohio 44240, USA}
\email{bjaye@kent.edu}\email{nazarov@math.kent.edu}
\thanks{Jaye supported in part by NSF DMS-1500881. Nazarov supported in part by NSF DMS-1265623.}
\date{\today}

\begin{abstract} We characterize the non-atomic measures $\mu$ for which all Calder\'{o}n-Zygmund operators with antisymmetric kernels of a fixed non-integer dimension $s$ are bounded in $L^2(\mu)$ in terms of a positive quantity, the Wolff energy.
\end{abstract}

\maketitle
\tableofcontents

\section{Introduction}

 Fix $d\geq 1$ and $s\in (0,d)$.  A smooth $s$-dimensional Calder\'{o}n-Zygmund (CZ) kernel is an odd function $K:\R^d\backslash\{0\}\rightarrow \R$ satisfying
$$|K(x)|\leq \frac{1}{|x|^s}\text{ and }|\nabla K(x)|\leq \frac{1}{|x|^{s+1}}
$$
for every $x\neq 0$.

Fix a non-atomic locally finite Borel measure $\mu$.  We say that all $s$-dimensional Calder\'{o}n-Zygmund operators (CZOs) are bounded in $L^2(\mu)$ if there is a constant $C>0$ such that for every CZ kernel $K$,
\begin{equation}\label{CZObdd}\sup_{\eps>0} \int_{\R^d}\Bigl|\int_{\R^d\backslash B(x,\eps)}K(x-y)f(y)d\mu(y)\Bigl|^2d\mu(x)\leq C\|f\|^2_{L^2(\mu)},
\end{equation}
for every $f\in L^2(\mu)$.\\

The purpose of this article is to give a characterization of those non-atomic measures for which all $s$-dimensional CZOs are bounded when $s\not\in \mathbb{Z}$.   We prove the following theorem:

\begin{thm}\label{thm1}  Fix $s\not\in \mathbb{Z}$.  Let $\mu$ be a non-atomic locally finite Borel measure.   Then all $s$-dimensional CZOs are bounded in $L^2(\mu)$ if and only if there is a constant $C>0$ such that
\begin{equation}\label{wolffenergy}\mathcal{W}_2(\mu, Q)\leq C\mu(Q)
\end{equation}
for every cube $Q\subset \R^d$, where $\mathcal{W}_2(\mu, Q)$ denotes the Wolff energy
$$\mathcal{W}_2(\mu, Q)=\int_{Q}\int_0^{\infty}\Bigl(\frac{\mu(Q\cap B(x,r))}{r^s}\Bigl)^2\frac{dr}{r}d\mu(x).
$$
\end{thm}

The most notable point in the characterization is that it gives the equivalence between the $L^2(\mu)$ boundedness of a collection of operators associated to odd (and so sign changing) kernels, and the $L^2(\mu)$ boundedness of a \emph{positive} (albeit non-linear) operator.\\

The sufficiency of the Wolff energy condition (\ref{wolffenergy}) for the boundedness of all CZOs was essentially proved in the paper by Mateu-Prat-Verdera \cite{MPV}, see also \cite{ENV}, and is not particularly subtle.  The proof consists of an elementary symmetrization trick and an application of the $T(1)$-theorem, see Appendix A of \cite{JN2} for a concise proof in the generality required for Theorem \ref{thm1}.  As such, we shall only be concerned with the statement that the boundedness of all CZOs implies that (\ref{wolffenergy}) holds.\\

It is of great interest whether the boundedness of the $s$-Riesz transform alone (the CZO with (vector valued) kernel $K(x) =\tfrac{x}{|x|^{s+1}}$) already implies that the Wolff energy condition (\ref{wolffenergy}) holds. This was shown to be the case when $s\in (0,1)$ by Mateu-Prat-Verdera \cite{MPV}, where the condition (\ref{wolffenergy}) was first explicitly introduced in relation to the boundedness of singular integral operators (a similar condition involving gauges of Hausdorff measures had previously appeared in Mattila's paper \cite{Mat1} on the analytic capacity of certain Cantor sets). \\

Together with Maria Carmen Reguera and Xavier Tolsa \cite{JNRT}, we recently showed that the Mateu-Prat-Verdera characterization for the boundedness of the Riesz transform continues to hold in the case $s\in (d-1,d)$.  It is an open problem whether the boundedness of the $s$-Riesz transform is equivalent to the Wolff energy condion (\ref{wolffenergy}) if $s\in (1,d-1)\backslash \mathbb{N}$.  An interesting intermediate problem would be to show that the $L^2(\mu)$ boundedness of all \emph{homogeneous} CZOs is enough to conclude that (\ref{wolffenergy}) holds.  \\


One can view Theorem \ref{thm1} as a non-integer variant of a theorem of David and Semmes \cite{DS}, which states that if $s\in \mathbb{Z}$ and $\mu$ is an Ahlfors-David regular measure, then all $s$-dimensional CZO's are bounded in $L^2(\mu)$ if and only if $\mu$ is uniformly rectifiable\footnote{In \cite{DS}, a slightly weaker notion of the boundedness in $L^2(\mu)$ of all CZOs is considered: For every CZ kernel $K$, the inequality (\ref{CZObdd}) holds for a constant $C=C_K$ that  may depend on the kernel. The proof of Theorem \ref{thm1} may be modified to obtain the same conclusion under this relaxed condition as well.}.  The analysis that follows shares quite a few similarities with the proof of the David-Semmes theorem.  Using their description of Ahlfors-David regular symmetric measures, Mattila and Preiss \cite{MP} showed that the David-Semmes theorem continues to hold if one only assumes that all operators with kernels of the form $\varphi(|x|)\tfrac{x}{|x|^{s+1}}$ are bounded in $L^2(\mu)$, where $\varphi\in C^{\infty}(0,\infty)$ satisfies $|\varphi^{(k)}(t)|\leq C_kt^{-k}$ for all $t>0$ and $k\geq 0$.\\ 

We remark that no complete analogue of Theorem \ref{thm1} is known in the case of integer dimension CZOs except for $d=2$, in which case one can refer to Tolsa's recent memoir \cite{Tol}.  An intriguing sufficient condition for the boundedness of all integer dimensional CZOs, involving  Jones' $\beta$-numbers, was recently found by Girela-Sarri\'{o}n \cite{G}, building upon previous work by Azzam-Tolsa \cite{AT}.  We wonder if the necessity of the Girela-Sarri\'{o}n condition could be proved by modifying the analysis carried out in this paper.\\

We follow the same high level scheme as in the aforementioned paper \cite{JNRT}, but the analysis is significantly simpler.   The boundedness of all CZO's allows one to argue, via a standard argument already used in \cite{DS}, that a certain collection of square functions is bounded.  More precisely, for every odd Lipschitz function $\varphi$ with compact support, we have that
$$\int_{\R^d}\sum_{k\in \mathbb{Z}} \Bigl|\int_{\R^d}\frac{\varphi(\tfrac{x-y}{2^k})}{2^{ks}} f(y)d\mu(y)\Bigl|^2d\mu(x)\leq C(\varphi)\|f\|^2_{L^2(\mu)},
$$
for every $f\in L^2(\mu)$.  By introducing (a simplified version of) the machinery used in \cite{JN2, JNRT}, we reduce matters to describing the structure of \emph{smoothly reflectionless measures}, these are measures $\mu$ for which the convolution $\varphi*\mu$ is constant on the support of $\mu$ for every compactly supported odd Lipschitz continuous function $\varphi$.  It turns out that smoothly reflectionless measures are rather easy to describe.  This comes in sharp contrast with reflectionless measures for the $s$-Riesz transform, where there are countless open problems.  In fact, one of our motivations for writing this paper is to give an accessible introduction to some of the mathematics in \cite{JNRT}, along with \cite{JN1,JN2}, without making the reader suffer through the technical difficulties.

\section{Preliminaries}

\subsection{Notation}

\begin{itemize}
\item A constant $C>0$ shall refer to a constant that may change from line to line.  Any constant may depend on $d$ and $s$ without mention.  If a constant depends on parameters other than $d$ and $s$, then these parameters are indicated in parentheses after the constant.
\item We denote the closure of a set $E$ by $\overline{E}$.
\item For $x\in \R^d$ and $r>0$, $B(x,r)$ denotes the open ball centred at $x$ with radius $r$.
\item By a measure, we shall always mean a non-negative locally finite Borel measure.
\item We denote by $\supp(\mu)$ the closed support of $\mu$, that is, $$\supp(\mu) = \mathbb{R}^d\backslash \bigl\{\cup B : B \text{ is an open ball with }\mu(B)=0\bigl\}.$$
\item For a closed set $E$, we shall denote by $\mu|E$ the restriction of the measure $\mu$ to $E$.
\item For $\beta\geq 0$, we denote by $\mathcal{H}^{\beta}$ the $\beta$-dimensional Hausdorff measure.
\item We set $\langle f,g\rangle_{\mu} = \int_{\R^d}fg\,d\mu$.
\item For a cube $Q\subset \R^d$, $\ell(Q)$ denotes its side-length. For $A>0$, we denote by $AQ$ the cube concentric to $Q$ of side-length $A\ell(Q)$.
\item We define the ratio of two cubes $Q$ and $Q'$ by
$$[Q':Q]=\Bigl|\log_2 \frac{\ell(Q')}{\ell(Q)} \Bigl|.$$
\item The density of a cube $Q$ (with respect to a measure $\mu$) is given by $\displaystyle D_{\mu}(Q)=\frac{\mu(Q)}{\ell(Q)^s}.$
\item For a set $U\subset \R^d$, we denote by $\Lip_0(U)$ the set of Lipschitz continuous functions on $\R^d$ that are compactly supported in the interior of $U$.  We define the homogeneous Lipschitz norm of $f\in \Lip_0(U)$ by $$\|f\|_{\Lip}= \sup_{x,y\in \R^d, \, x\neq y}\frac{|f(x)-f(y)|}{|x-y|}. $$
\item We say that a sequence of measures $\mu_k$ converges weakly to a measure $\mu$ if
$$\lim_{k\to\infty}\int_{\R^d}fd\mu_k = \int_{\R^d}fd\mu,
$$
for every $f\in C_0(\R^d)$ (the space of continuous functions on $\R^d$ with compact support).  We shall record some basic facts regarding weak convergence of measures, see for instance Chapter 1 of \cite{Mat} for details.  The weak limit enjoys the following two semi-continuity properties:
\begin{enumerate}
\item $\mu(U) \leq \liminf_{k\rightarrow \infty}\mu_k(U)$ for every open set  $U\subset \R^d$, and
\item $\mu(K) \geq \limsup_{k\rightarrow \infty}\mu_k(K)$  for every compact set $K\subset \R^d$.
\end{enumerate}
The separability of $C_0(\R^d)$ along with the Riesz representation theorem, yields the following weak compactness result: If $\mu_k$ is a sequence of measures such that $\sup_k \mu_k(B(0,R))<\infty$ for every $R>0$, then the sequence has a weakly convergent subsequence.
\end{itemize}

\subsection{The lattice of triples of dyadic cubes}\label{dyadictriples}


Let $\mathcal{D}=\mathcal{D}(\mathcal{Q})$ denote the lattice of concentric triples of open dyadic cubes from a dyadic lattice $\mathcal{Q}$.  Cubes in the lattice $\mathcal{D}$ are not disjoint on a given level, but have finite overlap.

Set $Q_0= 3(0,1)^d = (-1,2)^d$.  For a cube $Q\in \dy$, we set $\mathcal{L}_Q$ to be the canonical linear map (a composition of a dilation and a translation) satisfying $\mathcal{L}_Q(Q_0)=Q$.  

The cubes in $\mathcal{D}=\mathcal{D}(\mathcal{Q})$ have a natural family tree:  For instance, a cube $P\in \mathcal{D}$ is the grandparent of $Q\in \mathcal{D}$ if $P=3\underline{P}$ and $Q=3\underline{Q}$ where $\underline{P}\in \mathcal{Q}$ is the unique dyadic cube containing $\underline{Q}\in \mathcal{Q}$ with $\ell(\underline{P})=4\ell(\underline{Q})$.

\begin{lem}\label{dycubecontain} Suppose that $Q\in \mathcal{D}$, and $P$ is any cube that intersects $Q$  with $\ell(P)\leq \ell(Q)$.  Then the grandparent $\wt{Q}$ of $Q$ contains $P$ (in fact, $\wt{Q}\supset 3Q$).\end{lem}

\begin{proof} For $Q=3\underline{Q}$ with $\underline{Q}\in \mathcal{Q}$, the cube $\wt{Q}$ is the triple of a cube $\wt{\underline{Q}}$ that contains $\underline{Q}$ and satisfies $\ell(\wt{\underline{Q}})=4\ell(\underline{Q})$.  Therefore $\wt{Q}= 3\wt{\underline{Q}}\supset 9\underline{Q}=3Q\supset P$, and the lemma is proved.\end{proof}

We say that a sequence of lattices $\mathcal{D}_k$ stabilizes in a lattice $\mathcal{D}'$ if every $Q'\in \mathcal{D}'$ lies in $\mathcal{D}_k$ for sufficiently large $k$.

\begin{lem} Suppose $\mathcal{D}^{(k)}$ is a sequence of lattices with $Q_0\in \mathcal{D}^{(k)}$ for all $k$.  Then there exists a subsequence of the lattices that stabilizes to some lattice $\mathcal{D}'$.
\end{lem}

The lemma is proved via a diagonal argument:  For every $n\geq0$, there are $2^{nd}$ ways to choose a dyadic cube of sidelength $2^n$ so that $(0,1)^d$ is one of its dyadic descendants.\\

Finally, we remark that there is a constant $C>0$ such that for any lattice $\mathcal{D}$ and measure $\mu$,
$$\int_{0}^{\infty}\Bigl(\frac{\mu(B(x,r))}{r^s}\Bigl)^2\frac{dr}{r}\leq C\sum_{Q\in \mathcal{D}}D_{\mu}(Q)^2\chi_Q(x) \text{ for every }x\in \R^d,
$$
and therefore, by integrating both sides of this inequality with respect to $\mu$, we see that
$$\mathcal{W}_2(\mu, \R^d)\leq C\sum_{Q\in \mathcal{D}}D_{\mu}(Q)^2\mu(Q).
$$

\subsection{The growth condition}\label{growth}  Fix any $s\in (0,d)$ (integer or not).  If $\mu$ is a finite non-atomic measure for which all CZO's are bounded in $L^2(\mu)$, then necessarily $\sup_{Q\in \mathcal{D}}D_{\mu}(Q)<\infty$ for any lattice $\mathcal{D}$.  This is even true if one only considers the boundedness of certain non-degenerate CZOs (for instance the Riesz transform).  For a simple proof see Proposition 1.4 in Chapter 3 of David \cite{Dav}.

\section{From a CZO to a square function}\label{squaresec}

We first follow a rather standard path, already used by David-Semmes in \cite{DS}, to introduce a square function.

Suppose that $\mu$ is a finite measure for which all CZOs are bounded in $L^2(\mu)$.

For $M>0$, pick any odd function $\varphi\in \Lip_0(B(0,M))$ with $\|\varphi\|_{\Lip}\leq 1$.  Set $(\eps_n)_{n\in \mathbb{Z}}$ to be a sequence of independent mean zero $\pm1$-valued random variables (defined on some probability space $\Omega$).  Then, for any $n_0\in \mathbb{N}$ and $\omega\in \Omega$, notice that the odd function
$$K(x)=\sum_{n\in \mathbb{Z},\,|n|\leq n_0}\eps_n(\omega)\frac{1}{3^s\cdot 2^{ns}}\varphi\bigl(\frac{x}{2^n}\bigl)
$$
satisfies $$|K(x)|\leq \frac{C(M)}{|x|^s}\text{ and }|\nabla K(x)|\leq \frac{C(M)}{|x|^{s+1}}\text{ for }x\neq 0,$$  (the factor of $3^s$ is an artifact of using the lattice of triples).  Therefore, for some constant $C(M)$ (which we reiterate does not depend on $\varphi$ or $n_0$), we have that
$$\int_{\R^d}\Bigl|\sum_{|n|\leq n_0} \eps_n(\omega)\int_{\R^d} \frac{1}{3^s\cdot 2^{ns}}\varphi\bigl(\frac{x-y}{ 2^n}\bigl)f(y)d\mu(y)\Bigl|^2d\mu(x)\leq C(M)\|f\|^2_{L^2(\mu)},
$$
for every $f\in L^2(\mu)$. Taking the expectation over $\omega\in\Omega$, and using independence, we deduce that
$$\sum_{|n|\leq n_0} \|T_{\varphi,3\cdot 2^n}(f\mu)\|_{L^2(\mu)}^2\leq C(M)\|f\|^2_{L^2(\mu)},$$
where $$T_{\varphi, \ell}(f\mu)(x) = \int_{\R^d}\frac{1}{\ell^{s}}\varphi\Bigl(\frac{3(x-y)}{\ell}\Bigl)f(y)d\mu(y).$$
Now we may let $n_0\to \infty$ to conclude that
\begin{equation}\label{squarebd}\sum_{n\in \mathbb{Z},\; \ell=3\cdot 2^n} \|T_{\varphi,\ell}(f\mu)\|_{L^2(\mu)}^2\leq C(M)\|f\|^2_{L^2(\mu)} \text{ for every }f\in L^2(\mu).
\end{equation}

\section{The Riesz System}\label{rieszsec}

Suppose that $\mu$ is a finite measure for which all CZOs are bounded in $L^2(\mu)$ with operator norms at most $1$.  For any cube $Q$, it is trivial to observe that all CZOs are bounded in $L^2(\chi_Q\mu)$, and that the operator norms can only decrease.  Consequently, making reference to Section \ref{dyadictriples}, we conclude that in order to prove Theorem \ref{thm1}, it suffices to find a constant $C>0$ so that
$$\sum_{Q\in \dy} D_{\mu}(Q)^2\mu(Q)\leq C\mu(\R^d).
$$

For $A\gg 1$, and a cube $Q\in \mathcal{D}$, define the set
$$\Psi_{\mu}^A(Q) = \Bigl\{\psi\in \Lip_0(AQ): \;\|\psi\|_{\Lip}\leq \frac{1}{\ell(Q)}\text{ and }\int_{\R^d}\psi \,d\mu=0\Bigl\}.
$$
We note that the collection $\{\Psi_{\mu}^A(Q)\}_{Q\in \mathcal{D}}$ is a Riesz system in the sense that there is a constant $C(A)>0$ such that for any choices of functions $\psi_Q\in \Psi_{\mu}^A(Q)$, it holds that
\begin{equation}\label{Rieszbd}\sum_{Q\in \mathcal{D}}\frac{|\langle g, \psi_Q\rangle|^2}{\mu(3AQ)}\leq C(A)\|g\|_{L^2(\mu)}^2 \text{ for every }g\in L^2(\mu).\end{equation}
(See Appendix B of \cite{JN2} for the simple proof).

Combining this with (\ref{squarebd}) we see that for every odd function $\varphi\in \Lip_0(B(0,M))$ with $\|\varphi\|_{\Lip}\leq 1$, and for every $A>1$, there is a constant $C(M, A)$ such that for any choices of $\psi_Q\in \Phi_{\mu}^A(Q)$,
\begin{equation}\begin{split}\nonumber
\sum\limits_{Q\in \mathcal{D}}&\frac{|\langle T_{\varphi,\ell(Q)}(\mu),\psi_Q\rangle_{\mu}|^2}{\mu(3AQ)}\leq \sum\limits_{n\in \mathbb{Z},\,\ell=3\cdot 2^n}\sum_{Q\in \mathcal{D}}\frac{|\langle T_{\varphi,\ell}(\mu),\psi_Q\rangle_{\mu}|^2}{\mu(3AQ)}\\&\leq C(A)\sum\limits_{n\in \mathbb{Z},\,\ell=3\cdot 2^n}\| T_{\varphi,\ell}(\mu)\|_{L^2(\mu)}^2\leq C(M,A)\mu(\R^d).
\end{split}\end{equation}

Let us define the \emph{Lipschitz oscillation coefficient}
$$\Theta_{\mu,\varphi}^A(Q) = \sup\limits_{\psi\in \Psi_{\mu}^A(Q)}|\langle T_{\varphi,\ell(Q)}(\mu),\psi\rangle_{\mu}|.$$
Then we infer that
\begin{equation}\label{plugRiesz}
\sum\limits_{Q\in \mathcal{D}}\frac{\Theta_{\mu,\varphi}^A(Q)^2}{\mu(3AQ)}\leq C(M,A)\mu(\R^d).
\end{equation}

Now, let us take a countable dense (in the uniform metric) subset $(\varphi_j)_{j\in \mathbb{N}}$ of the separable space of odd functions in $\varphi\in\Lip_0(\R^d)$ with $\|\varphi\|_{\Lip}\leq 1$, arranged so that $\varphi_j\in \Lip_0(B(0,j))$.

We notice that \emph{if} (and it is a big if, as it is false) we could find some number $\varphi_1,\dots, \varphi_N$ of the functions, along with
 universal constants $\Delta>0$ and $A>1$ such that for every $Q\in \mathcal{D}$,
 \begin{equation}\label{nondegenerate}\max_{j\in \{1,\dots, N\}}\Theta_{\mu,\varphi_j}^A(Q)\geq \Delta D_{\mu}(Q)\mu(Q),
 \end{equation}
 then it would immediately follow from (\ref{plugRiesz}) that
 $$\sum_{Q\in \mathcal{D}}D_{\mu}(Q)^2 \frac{\mu(Q)}{\mu(3AQ)}\mu(Q)\leq \frac{C(N, A)}{\Delta^2}\mu(\R^d),
 $$
 which is essentially what we want to prove.  However, as we have already indicated, this is too good to be true, and there are in general cubes for which (\ref{nondegenerate}) fails.  For instance, if $\mu=\chi_{B(0,M)}m_d$, then the left hand side of (\ref{nondegenerate}) equals zero if $AQ\subset B(0, \tfrac{M}{2})$ and $N\ell(Q)\leq \tfrac{M}{2}$.

 We therefore modify our goal to the following rather more complicated (but achievable) statement:

 \begin{goal} Find absolute constants $N\in \mathbb{N}$, $A>1$, $\Delta>0$ and $c>0$, and a rule $\mathcal{F}$ that associates with each finite measure $\mu$ a set of cubes $\mathcal{F}(\mu)\subset \mathcal{D}$, such that the following two conditions hold.

 (\textbf{A})  (Large Lipschitz Oscillation Coefficient)  For each $Q\in \mathcal{F}(\mu)$
 $$\max_{j\in \{1,\dots,N\}}\Theta_{\mu, \varphi_j}^A(Q)\geq \Delta D_{\mu}(Q)\mu(Q).
 $$

 (\textbf{B}) (Large Portion of Wolff Potential)  If $\sup_{Q\in \mathcal{D}}D_{\mu}(Q)<\infty$, then
 $$\sum_{Q\in \mathcal{F}(\mu)}D_{\mu}(Q)^2 \frac{\mu(Q)}{\mu(3AQ)}\mu(Q)\geq c\sum_{Q\in \mathcal{D}}D_{\mu}(Q)^2\mu(Q).
 $$
  \end{goal}

 Once this goal is achieved, the main result would follow.  Indeed, if $\mu$ is a finite measure for which all CZO's are bounded in $L^2(\mu)$, then necessarily $\sup_{Q\in \mathcal{D}}D_{\mu}(Q)<\infty$ (Section \ref{growth}).  From property (\textbf{A}) and (\ref{plugRiesz}) we see that
 $$\sum_{Q\in \mathcal{F}(\mu)}D_{\mu}(Q)^2 \frac{\mu(Q)}{\mu(3AQ)}\mu(Q)\leq C(A,\Delta,N)\mu(\R^d),
 $$
and therefore the desired bound follows from property (\textbf{B}).

 We shall follow \cite{JNRT} in making the choice of the rule.  This calls for two refinement processes on the lattice $\mathcal{D}$.

\section{Upward Domination}

Fix $\eps>0$.  Fix a measure $\mu$.

\begin{defn}    We say that $Q'\in \dy$  \emph{dominates}  $Q\in \dy$  \emph{from above} if $Q'\supset Q$ and
$$D_{\mu}(Q')\geq 2^{\eps  [Q':Q]}D_{\mu}(Q).
$$
The set of those cubes in $\mathcal{D}$ that cannot be dominated from above by another cube in $\mathcal{D}$ is denoted by $\dysel(\mu)$ (or just $\dysel$).
\end{defn}

Notice that for any $M>1$, a cube $Q\in \dysel(\mu)$ is $M$-doubling in the sense that $\mu(MQ)\leq CM^{s+\eps}\mu(Q)$.  Indeed, as a consequence of Lemma \ref{dycubecontain} above, we may cover $MQ$ by a cube $Q'\in\dy$ that contains $Q$ and has side-length comparable to $M\ell(Q)$.  But then $\mu(MQ)\leq \mu(Q')\leq \bigl(\frac{\ell(Q')}{\ell(Q)}\bigl)^s2^{\eps[Q:Q']}\mu(Q)\leq CM^{s+\eps}\mu(Q)$.

\begin{lem}\label{upwolff}  Suppose that  $\sup_{Q\in \mathcal{D}}D_{\mu}(Q)<\infty$.  Then there exists a constant $c(\eps )>0$ such that
$$\sum_{Q\in \dysel(\mu)}D_{\mu}(Q)^2\mu(Q) \geq c(\eps) \sum_{Q\in \dy}D_{\mu}(Q)^2\mu(Q).
$$
\end{lem}

\begin{proof}

We first claim that every $Q\in \dy\backslash \dysel$ with $\mu(Q)>0$ can be dominated from above by a cube $\wt{Q}\in \dysel$.

Indeed, note that if $Q'$ dominates $Q$ from above, then certainly
$$[Q':Q]\leq\frac{1}{\eps}\log_2\Bigl(\frac{\sup_{Q''\in \mathcal{D}}D_{\mu}(Q'')}{D_{\mu}(Q)}\Bigl),$$ or else we would have that $D_{\mu}(Q')>\sup_{Q''\in \mathcal{D}}D_{\mu}(Q'')$ (which is absurd).  Consequently, there are only finitely many candidates for a cube that dominates $Q$ from above.  To complete the proof of the claim, choose $\wt{Q}\in \dy$ to be a cube of largest side-length that dominates $Q$ from above.  Then $\wt{Q}\in \dysel$ (domination from above is transitive).

For each fixed $P\in \dysel$, consider those $Q\in \dy\backslash \dysel$ with $\mu(Q)>0$ and $\wt{Q}=P$.  Then
\begin{equation}\begin{split}\nonumber\sum_{Q\in \dy\backslash \dysel: \, \widetilde{Q}=P}& D_{\mu}(Q)^2\mu(Q) = \sum_{m\geq 1} \sum_{\substack{Q\in \dy\backslash \dysel:\\\ell(Q)=2^{-m}\ell(P),\, \widetilde{Q}=P}} D_{\mu}(Q)^2\mu(Q)\\
&\leq \sum_{m\geq 1}2^{-2\eps  m}D_{\mu}(P)^2\Bigl[\sum_{\substack{Q\in \dy:\\\ell(Q)=2^{-m}\ell(P), Q\subset P}} \mu(Q)\Bigl]
\end{split}\end{equation}
The sum in square brackets is bounded by $C\mu(P)$, and so by summing over $P\in \dysel$, we see that
\begin{equation}\nonumber \sum_{Q\in \dy\backslash \dysel}D_{\mu}(Q)^2\mu(Q)\leq C(\eps)\sum_{P\in\dysel}D_{\mu}(P)^2\mu(P),\end{equation}
and the lemma is proved.
\end{proof}

\section{Downward Domination and the choice of the rule}

\begin{defn}  We say that $Q\in \dysel(\mu)$ \textit{is dominated from below by a (finite) bunch of cubes} $Q_j$ if the following conditions hold:
\begin{enumerate}
\item  $Q_j\in \dysel$,
\item $D_{\mu}(Q_j)\geq 2^{\eps  [Q:Q_j]}D_{\mu}(Q)$,
\item $3Q_j$ are disjoint,
\item $3Q_j\subset 3Q$,
\item $\displaystyle\sum_j D_{\mu}(Q_j)^22^{-2\eps [Q:Q_j]}\mu(Q_j) \geq D_{\mu}(Q)^2 \mu(Q).$
\end{enumerate}
We set $\dyselA=\dyselA(\mu)$ to be the set of all cubes $Q$ in $\dysel$ that cannot be dominated from below by a bunch of cubes except for the trivial bunch consisting of just the cube $Q$.
\end{defn}

\begin{lem}\label{downwolff}  Suppose that $\sup_{Q\in \mathcal{D}}D_{\mu}(Q)<\infty$.  There exists $c(\eps)>0$  such that
$$\sum_{Q\in \dyselA}D_{\mu}(Q)^2\mu(Q) \geq c(\eps) \sum_{Q\in \dysel}D_{\mu}(Q)^2\mu(Q).
$$
\end{lem}

\begin{proof}We start with a simple claim.

\textbf{Claim.}   Every $Q\in \dysel$ with $\mu(Q)>0$ is dominated from below by a bunch of cubes $P_{Q,j}$ in $\dyselA(\mu)$.

To prove the claim we make two observations. The first is transitivity: if the bunch $Q_1, \dots, Q_N$ dominates $Q'\in \dysel$ from below, and if (say) $Q_1$ is itself dominated from below by a bunch $P_1, \dots, P_{N'}$, then the bunch $P_1, \dots, P_{N'}$, $ Q_2, \dots, Q_N$ dominates $Q'$.  The second observation is that there are only finitely many cubes $Q'$ that can participate in a dominating bunch for $Q$:  Indeed, each such cube $Q'$ satisfies $D_{\mu}(Q')\geq 2^{\eps  [Q:Q']}D_{\mu}(Q)$, and so
$$[Q:Q']\leq\frac{1}{\eps}\log_2\Bigl(\frac{\sup_{Q''\in \mathcal{D}}D_{\mu}(Q'')}{D_{\mu}(Q)}\Bigl).$$

With these two observations in hand, we define a partial ordering on the finite bunches of cubes $(Q_j)_j$ that dominate $Q$ from below:  For two different dominating bunches $(Q^{(1)}_j)_j$ and $(Q^{(2)}_j)_j$, we say that $(Q^{(1)}_j)_j \prec (Q^{(2)}_j)_j$ if for each cube $3Q^{(1)}_j$, we have $3Q^{(1)}_j\subset 3Q^{(2)}_k$ for some $k$.  Since there are only finitely many cubes that can participate in a dominating bunch, there may be only finitely many different dominating bunches of $Q$, and hence there is a minimal (according to the partial order $\prec$) dominating bunch $(P_{Q,j})_j$.  Each cube $P_{Q,j}$ must lie in $\dyselA(\mu)$.



Now write
\begin{equation}\begin{split}\nonumber\sum_{Q\in \dysel}& D_{\mu}(Q)^2\mu(Q) \leq \sum_{Q\in \dysel}\sum_j D_{\mu}(P_{Q,j})^2\mu(P_{Q,j})2^{-2\eps [Q:P_{Q,j}]}\\
& \leq \sum_{P\in \dyselA}D_{\mu}(P)^2\mu(P)\Bigl[\sum_{Q: 3Q\supset 3P}2^{-2\eps [Q:P]}\Bigl].
\end{split}\end{equation}
The inner sum does not exceed $\tfrac{C}{\eps }$, and the lemma follows.\end{proof}

We may now define the rule $\mathcal{F}$.  For a measure $\mu$, set $\mathcal{F}(\mu) = \dyselA(\mu)$.  Then, for any $A>1$, the $3A$-doubling property of cubes in $\dysel$, along with Lemmas \ref{upwolff} and \ref{downwolff}, yield that property (\textbf{B}) holds, and the constant $c>0$ appearing in property (\textbf{B}) can be given in terms of $A$ and $\eps$.

It therefore remains to show that, for $A>1$ sufficiently large and $\eps>0$ sufficiently small, we can find $\Delta>0$ and $N\in \mathbb{N}$ so that property (\textbf{A}) holds.  We shall achieve this via a contradiction.

\section{The blow-up to a smoothly reflectionless measure}

Let us now fix $\eps>0$ so that the interval $(s-\eps,s+\eps)$ contains no integers.  We shall suppose that property $(\textbf{A})$ fails to hold for our rule $\mathcal{F}$.  That is, for every $k\in \mathbb{N}$, we can find a measure $\wt\mu_k$ and a cube $Q_k\in \dyselA(\wt\mu_k)$ for which
 \begin{equation}\nonumber\max_{j\in \{1,\dots, k\}}\Theta_{\wt\mu_k,\varphi_j}^k(Q_k)\leq \frac{1}{k} D_{\wt\mu_k}(Q_k)\wt\mu_k(Q_k).
 \end{equation}
We set $\mu_k = \frac{\wt\mu_k(\mathcal{L}_{Q_k}(\,\cdot\,))}{\wt\mu_k(Q_k)}$.  Notice that $\dy^{(k)} = \mathcal{L}^{-1}_{Q_k}\dy$ is a lattice containing $Q_0$.  Moreover, $Q_0 \in \dyselA^{(k)}(\mu_k)$, and
 \begin{equation}\label{degenerate}\max_{j\in \{1,\dots, k\}}|\langle T_{\varphi_j}(\mu_k),\psi\rangle_{\mu_k}|\leq \frac{1}{k} \text{ for every }\psi\in \Psi_{\mu_k}^k(Q_0),
 \end{equation}
where $T_{\varphi_j}(\mu_k)(x) = \int_{\R^d}\varphi_j(x-y)d\mu_k(y)$.

Insofar as $Q_0\in \dysel^{(k)}(\mu_k)$,  there is a constant $C>0$ such that for every $R>1$,
\begin{equation}\label{prelimgrowth}\mu_k(B(0,R))\leq CR^{s+\eps}.\end{equation}  Consequently, we may pass to a subsequence of the measures $\mu_k$ that converges weakly to a measure $\mu$.   Then $\mu(\overline{Q}_0)\geq 1$, and
\begin{equation}\label{limitedgrowth} \mu(B(0,R))\leq CR^{s+\eps}\text{ for every }R>1.\end{equation}

By passing to a further subsequence if necessary, we may assume that the lattices $\mathcal{D}^{(k)}$, which all contain $Q_0$, stabilize in a lattice $\mathcal{D}'$.

We next claim that, for each $j\in \mathbb{N}$,
\begin{equation}\label{limitzerodense}\langle T_{\varphi_j}(\mu),\psi\rangle_{\mu}=0
\end{equation}
for every $\psi\in \Lip_0(\R^d)$  with $\int_{\R^d}\psi d\mu=0$.

To see this, we first remark that, as $\mu_k$ converge to $\mu$ weakly, we have that $\mu_k\times\mu_k$ converge to $\mu\times\mu$ weakly over $C_0(\R^d\times\R^d)$. (For instance, one can show that finite linear combinations of functions of the form $(x,y)\mapsto f(x)g(y)$, with $f,g\in C_0(\R^d)$, are dense in $C_0(\R^d\times\R^d)$.)  Therefore $$\lim_{k\to\infty}\langle T_{\varphi_j}(\mu_k), \psi\rangle_{\mu_k}= \langle T_{\varphi_j}(\mu),\psi\rangle_{\mu},$$
for every $j\in \mathbb{N}$ and $\psi\in \Lip_0(\R^d)$ with $\int_{\R^d}\psi d\mu=0$.  Now, further assume that $\|\psi\|_{\Lip}<1$, and set $\psi_k = \psi - c_k\psi_0$, where $$c_k =\frac{\int_{\R^d}\psi d\mu_k}{\int_{\R^d}\psi_0 d\mu_k}\psi_0$$ for some function $\psi_0\in \Lip_0(\R^d)$ with $\int_{\R^d}\psi_0 d\mu=1$.  Then $c_k\rightarrow 0$ as $k\rightarrow \infty$, and so $\psi_k\in \Psi_{\mu_k}^k(Q_0)$ for sufficiently large $k$.  For those $k$, we get from (\ref{degenerate}) that $|\langle T_{\varphi_j}(\mu_k), \psi_k\rangle_{\mu_k}|\leq \frac{1}{k}$.  However, the uniformly restricted growth at infinity of the measures $\mu_k$, the property (\ref{prelimgrowth}), ensures that
$$\sup_k|\langle T_{\varphi_j}(\mu_k),\psi_0\rangle_{\mu_k}|\leq C(\varphi_j, \psi_0),
$$
so by writing $\langle T_{\varphi_j}(\mu_k), \psi_k\rangle_{\mu_k} = \langle T_{\varphi_j}(\mu_k), \psi\rangle_{\mu_k} - c_k\langle T_{\varphi_j}(\mu_k), \psi_0\rangle_{\mu_k}$, we infer that (\ref{limitzerodense}) holds for every $j$ and $\psi\in \Lip_0(\R^d)$ with $\int_{\R^d}\psi d\mu=0$ under the additional assumption that $\|\psi\|_{\Lip}<1$.  But this additional assumption can clearly be removed by considering $\tfrac{\psi}{\|\psi\|_{\Lip}+1}$ instead of $\psi$.

From the density of the sequence $(\varphi_j)_j$ in the collection of odd functions in the space $\Lip_0(\R^d)$ with Lipschitz norm at most $1$, we see that for each odd function $\varphi\in \Lip_0(\R^d)$, there exists $\Lambda_{\varphi}\in \R$ such that
$$T_{\varphi}(\mu) = \Lambda_{\varphi} \text{ on }\supp(\mu).
$$
We call such a measure \emph{smoothly reflectionless}.\\

To complete the proof of the property (\textbf{A}), and thereby conclude the proof of Theorem \ref{thm1}, it suffices to show that this limit measure $\mu$ cannot exist.  The properties that we have so far deduced about $\mu$ are

$\bullet$ $\mu(\overline{Q}_0)\geq 1$,

$\bullet$ $\mu(B(0,R))\leq R^{s+\eps}$  for every $R>1$, and

$\bullet$  $\mu$ is smoothly reflectionless.

However, there is no contradiction within these three properties:  we could have that $\mu = C\mathcal{H}^{\lfloor s\rfloor}|L$ for some $\lfloor s\rfloor$-plane $L$ and $C>0$.  On the other hand, we have not yet used the condition of the impossibility to dominate by a bunch from below.  We shall use the fact that $Q_0$ lies in $\dyselA^{(k)}(\mu_k)$ for every $k$ to prove an additional property of the limit measure $\mu$ which in particular ensures that it cannot be supported on (a countable collection of) $\lfloor s\rfloor$-dimensional planes.

\section{The weak density property of the limit measure}

We continue to work with the limit measure $\mu$ constructed in the previous section. For $T\gg 1$ consider the set
$$E_T = \Bigl\{x\in 2Q_0: \overline{D}_{\mu,\eps}(x)>T\Bigl\},
$$
where $\overline{D}_{\mu, \eps}(x) = \sup_{Q'\in\mathcal{D}':\,x\in Q'}D_{\mu}(Q')2^{-\eps[Q':Q_0]}$.

\begin{lem}\label{weakdens} There is a constant $C>0$ such that for all $T$ large enough,
$$\mu(E_T)\leq \frac{C}{T^2}.
$$
\end{lem}

Fix $m\in \mathbb{N}$, and consider the collection of cubes
\begin{equation}\begin{split}\nonumber\mathcal{D}'_{m, T} = \Bigl\{Q'\in \dy': \, Q'\cap 2Q_0\neq \varnothing, &\,\ell(Q')\in [2^{-m}, 2^m],\\ &\text{ and }D_{\mu}(Q')2^{-\eps[Q':Q_0]}>T\Bigl\}.
\end{split}\end{equation}
Since the lattices $\mathcal{D}^{(k)}$ stabilize, as long as $k$ is sufficiently large we have that every $Q'\in \mathcal{D}'_{m, T}$ lies in $\mathcal{D}^{(k)}$.  Also as $D_{\mu_k}(Q_0)=\tfrac{1}{3^s}$, we see that, provided $k$ is large enough,
$$D_{\mu_k}(Q')>T2^{\eps[Q':Q_0]}\geq T2^{\eps[Q':Q_0]}D_{\mu_k}(Q_0) \text{ for every }Q'\in \dy'_{m, T}.
$$

We begin with a simple auxiliary claim.

\begin{cla}\label{hatQclaim} Fix $T> 16^s2^{4\eps}$, and $k\in \mathbb{N}$.  For every $Q'\in \dy^{(k)}$ that intersects $2Q_0$ and satisfies
\begin{equation}\label{dkTdens}
D_{\mu_k}(Q')> T2^{\eps[Q':Q_0]}D_{\mu_k}(Q_0),
\end{equation}
we have that $\ell(Q')\leq \ell(Q_0)/4$, and so $3Q'\subset 3Q_0$.
\end{cla}

\begin{proof}[Proof of the claim]  Suppose to the contrary that $\ell(Q')\geq \ell(Q_0)/2$.  First note that, as a consequence of Lemma \ref{dycubecontain}, the grandparent $Q''$ of $Q'$ contains $3Q'$ and so in particular intersects $Q_0$.  But then as $\ell(Q'')\geq 2\ell(Q_0)$, Lemma \ref{dycubecontain} ensures that the grandparent $\wt{Q}'$ of $Q''$ contains $Q_0$.  Hence
$$D_{\mu_k}(\wt{Q}')\geq \frac{2^{-4\eps}}{16^s}2^{\eps[\wt{Q}': Q']}D_{\mu_k}(Q')\geq \frac{2^{-4\eps}}{16^s}T2^{\eps[\wt{Q}': Q_0]}D_{\mu_k}(Q_0).
$$
Under the condition on $T$, this contradicts the fact that $Q_0$, as a member of $\dyselA^{(k)}(\mu_k)$, cannot be dominated from above.  The claim is proved.
\end{proof}


\begin{proof}[Proof of the lemma]  It suffices to show that, for $T>16^s2^{4\eps}$, $$\mu_k\Bigl(\bigcup_{Q'\in \dy'_{m, T}}Q'\Bigl)\leq \frac{C}{T^2}$$ for sufficiently large $k$.  Fix $k$ large enough to ensure that $\dy'_{m,T}\subset \dy^{(k)}$ and (\ref{dkTdens}) holds for every $Q'\in \dy'_{m,T}$.  We claim that each $Q'\in \dy'_{m,T}$ is contained in some cube $\wh{Q}'$ with the following properties:
\begin{itemize}
\item  $D_{\mu_k}(\wh{Q}')> T2^{\eps[\wh{Q}': Q_0]}D_{\mu_k}(Q_0)$,
\item  $\wh{Q}'\in \dysel^{(k)}(\mu_k)$,
\item  $3\wh{Q}'\subset 3Q_0$.
\end{itemize}
 Indeed, for $Q'\in \dy'_{m, T}$, either $Q'\in \dysel^{(k)}(\mu)$ (in which case we can take  $\wh{Q}'=Q'$), or $Q'$ can be dominated from above by a cube $\wh{Q}'\in \dysel^{(k)}(\mu)$.   Either way, $\wh{Q}'$ satisfies the first two properties.  The third property now follows from Claim \ref{hatQclaim} applied to $\wh{Q}'\in \dy^{(k)}$.

Now, with the aid of the Vitali covering lemma, we choose a subcollection $(\wh{Q}'_j)_j$ with $3\wh{Q}'_j$ disjoint, and such that $15\wh{Q}'_j$ cover the union of the cubes in $\dy'_{m,T}$.  But then, since $\wh{Q}_j'\in \dysel^{(k)}(\mu_k)$, we have that $\mu_k(15\wh{Q}_j')\leq C\mu_k(\wh{Q}_j')$, and so,
\begin{equation}\begin{split}\nonumber\sum_j \mu_k(15\wh{Q}'_j)&\leq   C\sum_j \mu_k(\wh{Q}'_j)\\&\leq \frac{C}{T^2D_{\mu_k}(Q_0)^2}\sum_j D_{\mu_k}(\wh{Q}'_j)^22^{-2\eps[\wh{Q}'_j:Q_0]}\mu_k(\wh{Q}'_j),
\end{split}\end{equation}
but insofar as $Q_0\in \dyselA^{(k)}(\mu_k)$, the right hand side is at most $\frac{C}{T^2}\mu_k(Q_0)\leq \frac{C}{T^2}$.  The lemma follows.
\end{proof}




We shall not rely on the full strength of this lemma, as we shall only use it in conjunction with the following rather simple result.

\begin{lem}\label{nochargelowerdim} If $\nu$ is a measure satisfying
$$\nu\bigl(\bigl\{x\in 2Q_0: \overline{D}_{\nu,\eps}(x)>T\bigl\}\bigl)\rightarrow 0 \text{ as }T\to \infty,$$
then $\nu(L\cap \overline{Q}_0)=0$ for any $\lfloor s\rfloor$-plane $L$.\end{lem}

 This lemma can be seen as a consequence of standard theorems on the differentiation of measures (see for instance \cite{Mat}), but for the benefit of the reader we provide a direct proof.

\begin{proof}  Let $L$ be an $\lfloor s\rfloor$-plane.  Fix $T>0$.  Cover $L\cap \overline{Q}_0$ by at most $C\ell^{-\lfloor s \rfloor}$ cubes $Q_j\in \dy'$ of side-length $\ell =3\cdot 2^{-n}$, for $n\in \mathbb{N}$, $n\geq 2$.  If $D_{\nu}(Q_j)\leq T2^{\eps [Q_0:Q_j]}$, then $\nu(Q_j)\leq T3^{\eps}\ell^{s-\eps}$, and so the total measure of all such cubes $Q_j$ is at most $CT\ell^{s-\eps-\lfloor s \rfloor}$.  On the other hand, if $D_{\nu}(Q_j)\geq T2^{\eps [Q_0:Q_j]}$, then $Q_j\subset \bigl\{x\in 2Q_0: \overline{D}_{\nu,\eps}(x)>T\bigl\}$ (notice that $Q_j\cap \overline{Q}_0\neq \varnothing$ and $\ell(Q_j)< \tfrac{\ell(Q_0)}{2}$).  Therefore,
$$\nu(L\cap \overline{Q}_0)\leq CT\ell^{s-\eps-\lfloor s \rfloor}+\nu\bigl(\bigl\{x\in 2Q_0: \overline{D}_{\nu,\eps}(x)>T\bigl\}\bigl).
$$
Letting $\ell\rightarrow 0$ and then $T\rightarrow \infty$ yields that $\nu(L\cap \overline{Q}_0)=0$.\end{proof}


\section{The description of smoothly reflectionless measures}

The goal of this section is to derive a description of smoothly reflectionless measures.  A set of points $E\subset \R^d$ is said to be \emph{uniformly discrete} if there exists some $\delta>0$ such that  $|x-y|\geq \delta$ whenever $x,y\in E$, $x\neq y$.

\begin{prop}\label{description}  Suppose that $\mu$ is a smoothly reflectionless measure.  There exists a linear subspace $V$ of dimension $k\in \{0,\dots ,d\}$ along with a uniformly discrete set $E$ that is symmetric about each of its points (that is, if $x\in E$, and $y\in E$, then $2y-x\in E$), such that
$$\mu = \sum_{x\in E}f(x)\mathcal{H}^k|(V+x),
$$
where $f$ is a non-negative symmetric function on $E$ (symmetry here means that if $x,y\in E$, then $f(x) = f(2y-x)$).
\end{prop}

The proposition gives a complete description of smoothly reflectionless measures:  Any measure of the form $\sum_{x\in E}f(x)\mathcal{H}^k|(V+x)$, for a symmetric uniformly discrete set $E$, a  non-negative symmetric function $f$ on $E$, and a linear subspace $V$ of dimension $k$, is smoothly reflectionless.

Notice that the characteristic function of an (open) ball can be expressed as the monotone non-decreasing limit of a sequence of functions in $\Lip_0(\R^d)$.  Consequently, we see that if $\mu$ is smoothly reflectionless, then for every (open) ball $B$, there is a constant $\Lambda_B\in \R$ such that
$$\mu(x+B) - \mu(x-B) = [(\chi_{-B}-\chi_{B})*\mu](x)=\Lambda_B \text{ for all }x\in \supp(\mu).
$$

\begin{lem}\label{2ptlem}  Suppose that $\mu$ is smoothly reflectionless, and $x,y\in \supp(\mu)$.  With $z=y-x$, the points $x+kz$ are contained in $\supp(\mu)$ for all $k\in \mathbb{Z}$ and moreover $$\mu(B(x+2kz, r)) = \mu(B(x,r))\text{ whenever }k\in \mathbb{Z}\text{ and }r>0.$$
\end{lem}

\begin{proof}  First choose a sequence of radii $r_j\to 0^+$ with the property that the associated balls $B_j = B(z, r_j)$ satisfy either $\Lambda_{B_j}\geq 0$ for every $j$, or $\Lambda_{B_j}\leq 0$ for every $j$.  By relabelling $x$ and $y$ if necessary (thus replacing $B_j$ by $-B_j$), we may assume that $\Lambda_{B_j}\geq 0$ for every $j$.

Since $x\in \supp(\mu)$, we have $\mu(B(x,r_j))>0$ for every $j$. On the other hand, since $y\in \supp(\mu)$, and $\mu$ is smoothly reflectionless, we have
$$\mu(B(x+2z,r_j)) - \mu(B(x,r_j))=\mu(y+B(z, r_j)) - \mu(y-B(z,r_j)) = \Lambda_{B_j}.$$
Consequently $\mu(B(x+2z, r_j))\geq \mu(B(x, r_j))>0$ for every $j$ and so $x+2z\in \supp(\mu)$.  Repeating this argument with $x+2z=y+z$ playing the role of $y$, and $y=x+z$ playing the role of $x$, we get that $x+3z\in \supp(\mu)$.  Continuing in this fashion we see that $x+kz\in \supp(\mu)$ for all $k\in \mathbb{Z}_+$.

Now take a ball $B = B(kz, r)$ with $r>0$ and $k\in \mathbb{N}$.  Consider the constant $\Lambda_B\in \R$.  We claim that $\Lambda_B=0$.

 Let us suppose to the contrary that $\Lambda_B \neq 0$.  For $m\in \mathbb{N}$, notice that the ball $B^{(m)} = B(mkz,r)$ has reflectionless constant $\Lambda_{B^{(m)}} = m\Lambda_B$ for $m\in \mathbb{N}$.  To see this, merely write
\begin{equation}\begin{split}\nonumber\mu&(B(x+2mkz,r))-\mu(B(x,r)) \\&= \sum_{j=1}^m\bigl[ \mu(B(x+ 2jkz,r)) - \mu(B(x+2(j-1)kz,r))\bigl].\end{split}\end{equation}
 Using that $x+jz\in \supp(\mu)$ for all $j\in \mathbb{N}$, we see the left hand side of this identity equals $\Lambda_{B^{(m)}}$, while the right hand side equals $m\Lambda_B$. In the event that $\Lambda_B<0$, notice that the right hand side of the equality
\begin{equation}\label{mdiffequal}\mu(B(x+2mkz, r)) = m\Lambda_B +\mu(B(x,r)),\end{equation}
 can be made negative by choosing $m$ sufficiently large, which is absurd given that $\mu$ is a non-negative set function.  But if $\Lambda_B>0$  then using the reflectionless property at $x$, along with the identity $\Lambda_{B^{(2m)}}=2m\Lambda_B$, we obtain
$$\mu(B(x+2m k z,r)) - \mu(B(x-2m k z,r)) = 2m \Lambda_B,
$$
which, when combined with (\ref{mdiffequal}) yields that $\mu(B(x-2m k z,r)) = \mu(B(x,r))-m \Lambda_B$.  Then $\mu(B(x-2m k z,r))<0$ for large enough $m$, which is again absurd.  The claim is proved.

But now, one readily uses the reflectionless property at $x$ to deduce that $\mu(B(x-kz,r)) = \mu(B(x+kz,r))>0$ whenever $r>0$, and so $x-kz\in \supp(\mu)$ for $k\in \mathbb{N}$.  Finally, if $k\in \mathbb{Z}$ then we may use the reflectionless property at $x+kz$ to derive that $\mu(B(x+2kz, r)) = \mu(B(x,r))$ for any $r>0$.  The lemma is proved.
\end{proof}

\begin{lem}\label{reflplane}  Suppose that $\mu$ is a smoothly reflectionless measure.  If $V$ is a linear subspace, and $V+x_0\subset \supp(\mu)$ for some $x_0\in \supp(\mu)$, then for any $y\in \supp(\mu)$, we have that $V+y\subset \supp(\mu)$.
\end{lem}

\begin{proof} We may assume that $x_0=0$, so $V\subset\supp(\mu)$.  Lemma \ref{2ptlem}  ensures that the reflection of $y\in \supp(\mu)$ about each point in $V$ lies in $\supp(\mu)$.  As $v$ runs over $V$, the reflection of $y$ about $v$ runs over $V-y$.  Since then both $V$ and $V-y$ are contained in $\supp(\mu)$, we readily conclude from Lemma \ref{2ptlem} that $V+y\subset \supp(\mu)$.
\end{proof}

\begin{lem}\label{accumulation}  Suppose that $\mu$ is a smoothly reflectionless measure, and $V$ is a linear subspace with $V+x_0\subset \supp(\mu)$ for some $x_0\in \supp(\mu)$. If $\dist(x_0+V, \supp(\mu)\backslash (x_0+V))=0$, then $x_0+\operatorname{span}(V,e)\subset \supp(\mu)$ for some vector $e$ that is perpendicular to $V$.
\end{lem}

\begin{proof} We may assume that $x_0=0$.  Since $\dist(V, \supp(\mu)\backslash V)=0$, there is a sequence of points $(x_j)_{j\geq 1}$ in $\supp(\mu)\backslash V$ with $\dist(x_j, V)\rightarrow 0$ as $j\rightarrow \infty$.  Consider the closest point $z_j$ on $V$ to $x_j$, and set $d_j = |x_j-z_j|$.   By passing to a subsequence if necessary we may assume that the unit vectors $e_j=\tfrac{x_j-z_j}{d_j}$ converge to a unit vector $e$ that is perpendicular to $V$. Lemma \ref{reflplane} ensures that $V+x_j\subset \supp(\mu)$, and so in particular $d_je_j\in \supp(\mu)$.  As $0\in \supp(\mu)$, Lemma \ref{2ptlem} guarantees that $kd_je_j\in \supp(\mu)$ for every $k\in \mathbb{Z}$.  But then from Lemma \ref{reflplane}  we infer that $V+kd_je_j \subset\supp(\mu)$ for all $k\in \mathbb{Z}$.  Since $d_j\to 0$ and $e_j\to e$ as $j\to\infty$, we conclude that $v+\lambda e\in \supp(\mu)$ for every $v\in V, \, \lambda\in \R$.
\end{proof}

\begin{proof}[Proof of Proposition \ref{description}]  Without loss of generality, we may assume that $0\in \supp(\mu)$. Set $V$ to be the linear subspace of maximal dimension for which $V\subset \supp(\mu)$.  Lemma \ref{accumulation} then ensures that $V$ is isolated: there exists $\delta>0$ such that $\dist(V, \supp(\mu)\backslash V)\geq \delta$.


Consider the set $E$ of points that lie in $\supp(\mu)$ and are orthogonal to $V$ (this set includes $0$).  We next claim that if $x,y\in E$, $x\neq y$, then $|x-y|\geq \delta/2$.  Indeed, if we can find $x,y\in E$ with $x\neq y$ and $|x-y|<\delta/2$, then from Lemma \ref{2ptlem} we infer that the reflection of $0$ about $x$, the point $2x$, lies in $\supp(\mu)$.  Reflecting $2x$ about $y$ yields that $2(y -x)\in \supp(\mu)$.  But this is impossible because $\dist(0, \supp(\mu)\backslash V)\geq \delta$.  Thus $E$ is a uniformly discrete set that, in accordance with Lemma \ref{2ptlem}, is symmetric about each of its points.  Notice that the planes $(V+x)_{x\in E}$ are pairwise $\tfrac{\delta}{2}$-separated.  We now claim that $$\supp(\mu) = \bigcup_{x\in E}(V+x).$$
Clearly, Lemma \ref{reflplane} guarantees that $\bigcup_{x\in E}(V+x)\subset \supp(\mu)$.   To see the opposite inclusion, note that if $y\in \supp(\mu)$, then $V+y\subset \supp(\mu)$ (Lemma \ref{reflplane} again).  Consequently, if $x$ denotes the closest point to $0$ on $V+y$, then $x$ lies in $E$, and $y\in V+x$.

The second assertion of Lemma \ref{2ptlem} ensures that whenever $r>0$, and $x,y\in E$, \begin{equation}\label{fsymmetry}\mu(B(z,r))=\mu(B(z',r))\text{ if }z\in V+x\text{ and }z'\in V+(2y-x).\end{equation}  It readily follows that for $x\in E$, the measure $\mu(B(z,r))$ does not depend on $z\in V+x$.  By the uniqueness of measures that are uniformly distributed on small balls (for instance see the proof of Theorem 3.4 in \cite{Mat}), we derive that $\mu|(V+x)$ equals $f(x)\mathcal{H}^{k}|(V+x)$ for some $f(x)\geq 0$, where $k$ is the dimension of $V$.  The symmetry of the function $f$ follows immediately from (\ref{fsymmetry}).
\end{proof}

\section{Contradiction}

To complete the verification that property (\textbf{A}) holds for the rule $\mathcal{F}$, we are required to show that there cannot exist a measure $\mu$ satisfying the following properties:
\begin{enumerate}
\item $\mu$ is smoothly reflectionless,
\item $\mu(\overline{Q}_0)\geq 1$,
\item $\mu(B(0,R))\leq CR^{s+\eps}$ for any $R\geq 1$, and
\item for sufficiently large $T$, $\mu\bigl(\bigl\{x\in 2Q_0: \overline{D}_{\mu, \eps}(x)\geq T\bigl\}\bigl)\leq \frac{C}{T^2}$.
\end{enumerate}

However, Proposition \ref{description} we have that $$\mu = \sum_{x\in E}f(x)\mathcal{H}^k|(V+x),
$$
for some $k$-dimensional linear subspace $V$, a uniformly discrete set $E$, and some non-negative function $f$ on $E$. 
Since the assumption of Lemma \ref{nochargelowerdim} is satisfied, and $\mu(\overline{Q}_0)\geq 1$, we have that $k>\lfloor s\rfloor$.   On the other hand, the growth condition (3) ensures that $k<\lfloor s\rfloor+1$.  This contradiction completes the proof of Theorem \ref{thm1}.





\begin{thebibliography}{BLAHBL}


\bibitem[AT]{AT} J. Azzam and X. Tolsa, \emph{Characterization of n-rectifiability in terms of Jones' square function: Part II.} Geom. Funct. Anal. \textbf{25} (2015), no. 5, 1371--1412.

\bibitem[CPT]{CPT} V. Chousionis, L. Prat, and X. Tolsa, \emph{Square functions of fractional homogeneity and Wolff potentials}. Int. Math. Res. Notices (to appear). arXiv:1410.5272.

\bibitem[Dav]{Dav} G. David, \emph{Wavelets and Singular Integrals on Curves and Surfaces}. Lecture Notes in Math, Vol.
    \textbf{1465}, 1991.

\bibitem[DS]{DS}  G. David and S. Semmes, \emph{Singular integrals and rectifiable sets in $\mathbb{R}^n$: Beyond Lipschitz graphs.} Astérisque No. 193 (1991), 152 pp.


\bibitem[ENV]{ENV} V. Eiderman, F. Nazarov, A. Volberg, \emph{Vector-valued Riesz potentials: Cartan-type estimates and related capacities}. Proc. Lond. Math. Soc. (3) \textbf{101} (2010), no. 3, 727--758, arXiv:0801.1855.

\bibitem[G]{G} D. Girela-Sarri\'{o}n, \emph{Geometric conditions for the $L^2$-boundedness of singular integral operators with odd kernels with respect to measures with polynomial growth in $R^d$.} Preprint (2015).  arXiv:1505.07264.


\bibitem[JN1]{JN1} B. Jaye and F. Nazarov, \emph{Reflectionless measures for Calder\'{o}n-Zygmund Operators I: Basic Theory}. Preprint (2014). arXiv:1409.8556.

\bibitem[JN2]{JN2} B. Jaye and F. Nazarov, \emph{Reflectionless measures for Calder\'{o}n-Zygmund operators II: Wolff potentials and rectifiability.} Preprint (2015).  arXiv:1507.08329.

\bibitem[JNRT]{JNRT}  B. Jaye, F. Nazarov, M. C. Reguera and X. Tolsa, \emph{The Riesz transform of codimension smaller than one and the Wolff energy.} Preprint (2016). arXiv:1602.02821.

\bibitem[MT]{MT} A. Mas and X. Tolsa  \emph{Variation for the Riesz transform and uniform rectifiability}.   J. Eur. Math. Soc. \textbf{16} (11) (2014), 2267--2321.

\bibitem[Mat]{Mat} P. Mattila, \emph{Geometry of sets and measures in Euclidean spaces. Fractals and rectifiability}.
    Cambridge Studies in Advanced Mathematics, \textbf{44}. Cambridge University Press, Cambridge, 1995.


\bibitem[Mat1]{Mat1} P. Mattila, \emph{On the analytic capacity and curvature of some Cantor sets with
non-$\sigma$ finite length}. Publ. Mat. \textbf{40} (1996), 195--204.

\bibitem[MP]{MP}P. Mattila and D. Preiss, \emph{Rectifiable measures in $\mathbb{R}^n$ and existence of principal values for singular integrals.} J. London Math. Soc. (2) \textbf{52} (1995), no. 3, 482--496.

\bibitem[MPV]{MPV} J. Mateu, L. Prat, J. Verdera,
\emph{The capacity associated to signed Riesz kernels, and Wolff potentials.} J. Reine Angew. Math. \textbf{578} (2005), 201--223. arXiv:math/0411441.



\bibitem[MV]{MV} S. Mayboroda and A. Volberg, \emph{Finite square function implies integer dimension.} C. R. Math. Acad. Sci. Paris \textbf{347} (2009), no. 21--22, 1271--1276.

\bibitem[RT]{RT} M. C. Reguera and X. Tolsa, \emph{Riesz transforms of non-integer homogeneity on uniformly disconnected sets}. (To appear in Trans. Amer. Math. Soc.) arXiv:1402.3104.

\bibitem[Tol]{Tol} X. Tolsa, \emph{Rectifiable measures, square functions involving densities, and the Cauchy transform.} Preprint (2014). To appear in Mem. Amer. Math. Soc.

\bibitem[Vih]{Vih} M. Vihtil\"{a}, \emph{The boundedness of Riesz s-transforms of measures in $\mathbb{R}^n$.} Proc. Amer. Math. Soc. \textbf{124} (1996), no. 12, 3797--3804.


\bibitem[Vol]{Vol} A. Volberg, \emph{Calder\'{o}n-Zygmund capacities and operators on nonhomogeneous spaces.} CBMS Regional Conference Series in Mathematics, \textbf{100}, 2003.
\end{thebibliography}
 \end{document}